\documentclass[11pt,a4paper]{amsart}
\usepackage[utf8]{inputenc}
\usepackage[T1]{fontenc}
\usepackage{amsmath, amsthm, amssymb, mathrsfs}
\usepackage{geometry}
\usepackage{hyperref}

\newtheorem{theorem}{Theorem}[section]
\newtheorem{lemma}[theorem]{Lemma}
\newtheorem{proposition}[theorem]{Proposition}
\newtheorem{corollary}[theorem]{Corollary}
\theoremstyle{definition}
\newtheorem{definition}[theorem]{Definition}
\newtheorem{example}[theorem]{Example}
\newtheorem{remark}[theorem]{Remark}

\newcommand{\g}{\mathfrak{g}}
\newcommand{\h}{\mathfrak{h}}
\newcommand{\vol}{\mathrm{vol}}
\newcommand{\R}{\mathbb{R}}
\newcommand{\ad}{\mathrm{ad}}
\newcommand{\tr}{\mathrm{tr}}

\title[Rigidity and Affine Deformations on Contact Lie Algebras]{Rigidity and Affine Deformations of Foliations on Nilpotent and Solvable Contact Lie Algebras}

\author{Cheikh Khoule}
\address{Facult\'e des Sciences et Techniques, Universit\'e Cheikh Anta Diop de Dakar, Dakar, S\'en\'egal}
\email{cheikh.khoule@ucad.edu.sn}

\author{Rafik Khalfi}
\address{Facult\'e des Sciences de Bizerte, Universit\'e de Carthage, Bizerte, Tunisie}
\email{rafik.khalfi@fsb.ucar.tn}


\begin{document}

\begin{abstract}
Let $\g$ be a Lie algebra of dimension $2n+1$, let $\beta$ be a contact form, and let $\alpha$ be a nonzero integrable $1$-form. We study when the affine path $\alpha_t = (1-t)\alpha + t\beta$ consists of contact forms with the orientation determined by $\beta \wedge (d\beta)^n$. We first establish an algebraic criterion: affine deformability is equivalent to the non-negativity of a scalar obtained from the interaction form $\Theta(\alpha, \beta) = \alpha \wedge (d\beta)^n + n\beta \wedge d\alpha \wedge (d\beta)^{n-1}$. For unimodular Lie algebras this scalar is $(n+1)\alpha(Z_\beta)$, where $Z_\beta$ is the Reeb vector. We also derive a formula involving the modular character $\tau(X) = \tr(\ad_X)$ in the non-unimodular case. On nilpotent contact Lie algebras, every integrable $1$-form is closed; using the standard structural properties of nilpotent contact Lie algebras, this implies that every left-invariant foliation is an affine origin for every fixed contact form. An explicit three-dimensional example illustrates that the modular term can yield affine deformability even when $\alpha(Z_\beta) = 0$. We discuss the implications for left-invariant forms on Lie groups and compact quotients. We do not claim a general obstruction example in the non-unimodular case; establishing such an example requires a separate verified construction.
\end{abstract}

\maketitle

\section{Introduction}

The transition between integrable hyperplane distributions and contact structures is a classical theme in contact geometry, closely related to confoliation theory; see Eliashberg and Thurston \cite{EliashbergThurston1998}. Dathe and Khoule \cite{dathe2012deformations} studied affine deformations of codimension-one foliations into contact structures on manifolds. The present paper considers the corresponding left-invariant, algebraic problem on Lie algebras.

Let $\g$ be a real Lie algebra of dimension $2n+1$. A contact form $\beta \in \g^*$ determines a volume form $\vol_\beta = \beta \wedge (d\beta)^n$. A nonzero $1$-form $\alpha$ is integrable if $\alpha \wedge d\alpha = 0$. We ask when $\alpha$ can be joined to $\beta$ by an affine path of contact forms whose contact volume has the orientation prescribed by $\vol_\beta$.

The main algebraic result is a necessary and sufficient criterion in terms of the interaction scalar defined below. In the unimodular case this scalar is determined by the Reeb contraction $\alpha(Z_\beta)$. In the non-unimodular case it includes a term involving the modular character. We then apply the criterion to nilpotent Lie algebras, where integrability forces closedness.

Contact nilpotent Lie algebras and their structure have been studied, among others, by Alvarez, Rodr\'iguez-Vallarte, and Salgado \cite{alvarez2017contact}. Contact structures on Lie groups are also considered by Diatta \cite{diatta2008left}; Diatta and Foreman discuss lattices and five-dimensional contact solvmanifolds in \cite{diatta2015lattices}. These references provide context for our examples, but the classification statements below are deliberately limited to the explicitly computed cases.

\section{The affine criterion}

\subsection{Basic definitions}
Let $\g$ be a real Lie algebra of dimension $2n+1$. We use the Chevalley--Eilenberg convention
\[
    d\alpha(X, Y) = -\alpha([X, Y]), \quad \alpha \in \g^*.
\]
A form $\beta \in \g^*$ is a contact form if $\vol_\beta = \beta \wedge (d\beta)^n \neq 0$. Its Reeb vector $Z_\beta$ is characterized by $\beta(Z_\beta) = 1$ and $i_{Z_\beta} d\beta = 0$.

A nonzero $\alpha \in \g^*$ is integrable if $\alpha \wedge d\alpha = 0$. We write $\Omega^1_{\mathrm{int}}(\g)$ for the set of such forms.

\begin{lemma} \label{lem:dalpha_k}
If $\alpha \in \Omega^1_{\mathrm{int}}(\g)$, then $(d\alpha)^k = 0$ for every integer $k \geq 2$.
\end{lemma}
\begin{proof}
The condition $\alpha \wedge d\alpha = 0$ implies, by elementary exterior algebra, that $d\alpha = \alpha \wedge \eta$ for some $\eta \in \g^*$. Hence $(d\alpha)^2 = 0$, and therefore all higher powers vanish as well.
\end{proof}

\begin{definition}[Affine origin]
Let $\beta$ be a contact form on $\g$. A nonzero integrable form $\alpha$ is an \emph{affine origin} for $\beta$ if there are smooth functions $C, B : [0, 1] \to \R$ such that
\begin{gather*}
    C(0) = 1, \quad B(0) = 0, \quad C(1) = 0, \quad B(1) = 1, \\
    C(t) > 0, \quad B(t) > 0 \quad (0 < t < 1),
\end{gather*}
and, for $\alpha_t = C(t)\alpha + B(t)\beta$,
\[
    \alpha_t \wedge (d\alpha_t)^n > 0 \quad \text{for all } t \in (0, 1],
\]
where positivity is measured with respect to the orientation defined by $\vol_\beta$. The set of affine origins for $\beta$ is denoted by $\mathcal{A}(\beta)$.
\end{definition}

\begin{definition}[Interaction form and scalar]
For an integrable $\alpha$ and a contact form $\beta$, define
\[
    \Theta(\alpha, \beta) = \alpha \wedge (d\beta)^n + n\beta \wedge d\alpha \wedge (d\beta)^{n-1}.
\]
There is a unique scalar $f_{\alpha, \beta} \in \R$ such that
\[
    \Theta(\alpha, \beta) = f_{\alpha, \beta} \vol_\beta.
\]
\end{definition}

\begin{theorem}[Affine deformability criterion] \label{thm:criterion}
Let $\alpha$ be a nonzero integrable $1$-form and $\beta$ a contact form on a Lie algebra of dimension $2n+1$. For every smooth affine profile $\alpha_t = C(t)\alpha + B(t)\beta$,
\begin{equation} \label{eq:expansion}
    \alpha_t \wedge (d\alpha_t)^n = B(t)^n \left[ B(t) + C(t)f_{\alpha, \beta} \right] \vol_\beta.
\end{equation}
Moreover, $\alpha \in \mathcal{A}(\beta)$ if and only if $f_{\alpha, \beta} \geq 0$.
\end{theorem}
\begin{proof}
Since $d\alpha_t = C(t)d\alpha + B(t)d\beta$, Lemma \ref{lem:dalpha_k} gives
\[
    (d\alpha_t)^n = B(t)^n (d\beta)^n + n C(t)B(t)^{n-1} d\alpha \wedge (d\beta)^{n-1}.
\]
Wedging with $C(t)\alpha + B(t)\beta$ and using $\alpha \wedge d\alpha = 0$ proves \eqref{eq:expansion}.

If $f_{\alpha, \beta} \geq 0$, the profile $C(t) = 1-t$, $B(t) = t$ gives $B(t) + C(t)f_{\alpha, \beta} > 0$ for every $t > 0$, so $\alpha \in \mathcal{A}(\beta)$. Conversely, if $f_{\alpha, \beta} < 0$, continuity and $C(0) = 1$, $B(0) = 0$ imply $B(t) + C(t)f_{\alpha, \beta} < 0$ for all sufficiently small positive $t$. This contradicts the required positivity of the contact volume. Thus $f_{\alpha, \beta} \geq 0$ is necessary.
\end{proof}

\begin{remark}
The criterion is pointwise and algebraic. It does not require compactness of a manifold or an integration argument.
\end{remark}

\section{Integrable forms on nilpotent Lie algebras}

\begin{theorem}[Rigidity of nilpotent foliations] \label{thm:rigidity}
Let $\g$ be a finite-dimensional nilpotent Lie algebra. Every integrable $1$-form $\alpha \in \g^*$ is closed.
\end{theorem}
\begin{proof}
The result is immediate for $\alpha = 0$, so suppose $\alpha \neq 0$ and put $\h = \ker \alpha$. If $X, Y \in \h$, choose $V \in \g$ with $\alpha(V) = 1$. Evaluating $\alpha \wedge d\alpha = 0$ on $(V, X, Y)$ gives $d\alpha(X, Y) = 0$. Thus $\h$ is a codimension-one subalgebra.

Every proper subalgebra of a finite-dimensional nilpotent Lie algebra is strictly contained in its normalizer (the normalizer property; see \cite[Section 3.2]{humphreys1972lie}). Since $\h$ has codimension one, its normalizer is all of $\g$, so $\h$ is an ideal. For $X \in \h$, it follows that $[V, X] \in \h$, hence $d\alpha(V, X) = 0$. Together with $d\alpha(X, Y) = 0$ for $X, Y \in \h$, this proves $d\alpha = 0$.
\end{proof}

\section{The unimodular case and nilpotent contact algebras}

A Lie algebra is unimodular if its modular character $\tau \in \g^*$, defined by $\tau(X) = \tr(\ad_X)$, vanishes identically.

\begin{lemma} \label{lem:unimodular_d}
If $\g$ is unimodular of dimension $2n+1$, then $d\omega = 0$ for every $\omega \in \Lambda^{2n}\g^*$.
\end{lemma}
\begin{proof}
Fix a nonzero top-degree form $\Omega$. Every $2n$-form has the form $i_X \Omega$ for a unique $X \in \g$. Cartan's formula gives $d(i_X \Omega) = \mathcal{L}_X \Omega = -\tr(\ad_X)\Omega = 0$, using the Chevalley--Eilenberg convention fixed above.
\end{proof}

\begin{proposition} \label{prop:unimodular_theta}
Let $\g$ be a unimodular contact Lie algebra with contact form $\beta$. For every integrable $\alpha$,
\[
    \Theta(\alpha, \beta) = (n+1)\alpha \wedge (d\beta)^n = (n+1)\alpha(Z_\beta)\vol_\beta.
\]
In particular, $f_{\alpha, \beta} = (n+1)\alpha(Z_\beta)$.
\end{proposition}
\begin{proof}
Apply Lemma \ref{lem:unimodular_d} to $\omega = \alpha \wedge \beta \wedge (d\beta)^{n-1}$. The Leibniz rule gives
\[
    0 = d\omega = \beta \wedge d\alpha \wedge (d\beta)^{n-1} - \alpha \wedge (d\beta)^n.
\]
Substitution in the definition of $\Theta$ yields the first equality. The second follows from $i_{Z_\beta} d\beta = 0$ and $\beta(Z_\beta) = 1$.
\end{proof}

\begin{theorem}[Affine deformability on nilpotent contact Lie algebras] \label{thm:nilpotent_deform}
Let $\g$ be a nilpotent contact Lie algebra with contact form $\beta$. Assume the standard structural property that the Reeb vector $Z_\beta$ belongs to $[\g, \g]$ (in particular, this follows from the usual structure theorem for nilpotent contact Lie algebras; see \cite{alvarez2017contact}). Then every nonzero integrable $1$-form $\alpha$ is an affine origin for $\beta$.
\end{theorem}
\begin{proof}
By Theorem \ref{thm:rigidity}, $d\alpha = 0$, so $\alpha$ vanishes on $[\g, \g]$. Hence $\alpha(Z_\beta) = 0$. Nilpotent Lie algebras are unimodular, so Proposition \ref{prop:unimodular_theta} gives $f_{\alpha, \beta} = 0$. The criterion in Theorem \ref{thm:criterion} now implies $\alpha \in \mathcal{A}(\beta)$.
\end{proof}

\begin{corollary}
The conclusion of Theorem \ref{thm:nilpotent_deform} applies, in particular, to every filiform contact Lie algebra.
\end{corollary}

\begin{remark}
The structural input about $Z_\beta$ should be checked against the precise statement and hypotheses of the cited reference in the final bibliography. The proof above isolates exactly where that input is used.
\end{remark}

\section{The modular formula}

\begin{lemma} \label{lem:modular_d}
For a Lie algebra $\g$ of dimension $2n+1$ with modular character $\tau$, every $\omega \in \Lambda^{2n}\g^*$ satisfies
\[
    d\omega = -\tau \wedge \omega.
\]
\end{lemma}
\begin{proof}
Write $\omega = i_X \Omega$ for a nonzero volume form $\Omega$. As above, $d\omega = \mathcal{L}_X \Omega = -\tau(X)\Omega$. Since $\tau \wedge i_X \Omega = \tau(X)\Omega$, the stated identity follows.
\end{proof}

\begin{theorem}[General interaction formula] \label{thm:general_interaction}
Let $\beta$ be a contact form on a Lie algebra $\g$ of dimension $2n+1$, and let $\alpha$ be integrable. Then
\begin{equation} \label{eq:theta_modular}
    \Theta(\alpha, \beta) = (n+1)\alpha \wedge (d\beta)^n - n\tau \wedge \alpha \wedge \beta \wedge (d\beta)^{n-1}.
\end{equation}
Equivalently,
\[
    f_{\alpha, \beta} = (n+1)\alpha(Z_\beta) - n \frac{\tau \wedge \alpha \wedge \beta \wedge (d\beta)^{n-1}}{\vol_\beta},
\]
where the quotient denotes the unique scalar relating the two top-degree forms.
\end{theorem}
\begin{proof}
Apply Lemma \ref{lem:modular_d} to $\omega = \alpha \wedge \beta \wedge (d\beta)^{n-1}$. By the Leibniz rule,
\[
    d\omega = \beta \wedge d\alpha \wedge (d\beta)^{n-1} - \alpha \wedge (d\beta)^n.
\]
Comparing this with $d\omega = -\tau \wedge \alpha \wedge \beta \wedge (d\beta)^{n-1}$ gives
\[
    \beta \wedge d\alpha \wedge (d\beta)^{n-1} = \alpha \wedge (d\beta)^n - \tau \wedge \alpha \wedge \beta \wedge (d\beta)^{n-1}.
\]
Substitute this identity into the definition of $\Theta$ to obtain \eqref{eq:theta_modular}. The scalar formula follows by dividing by $\vol_\beta$.
\end{proof}

\begin{remark}
In the non-unimodular case the interaction scalar is not, in general, determined solely by $\alpha(Z_\beta)$. The example below shows that the modular term can yield $f_{\alpha, \beta} > 0$ when $\alpha(Z_\beta) = 0$. We do not assert here that the modular term can always produce an obstruction when $\alpha(Z_\beta) > 0$; that stronger claim requires a separate explicit example.
\end{remark}

\section{Examples}

\subsection{The five-dimensional Heisenberg algebra}

\begin{example}
Let $\mathfrak{h}_5$ have basis $e_1, \dots, e_5$ and nonzero brackets
\[
    [e_1, e_3] = e_5, \quad [e_2, e_4] = e_5.
\]
For $\beta = e_5^*$, we have $d\beta = -e_1^* \wedge e_3^* - e_2^* \wedge e_4^*$ and $\beta \wedge (d\beta)^2 \neq 0$. The Reeb vector is $Z_\beta = e_5$. Since $\mathfrak{h}_5$ is nilpotent, every integrable $\alpha$ is closed and annihilates $[\mathfrak{h}_5, \mathfrak{h}_5] = \R e_5$. Thus $\alpha(Z_\beta) = 0$, and Theorem \ref{thm:nilpotent_deform} gives $f_{\alpha, \beta} = 0$. Every nonzero left-invariant foliation is therefore an affine origin for $\beta$.
\end{example}

\subsection{A unimodular solvable example}

\begin{example}[A five-dimensional solvable algebra]
Consider the Lie algebra with nonzero brackets
\[
    [e_1, e_4] = e_1, \quad [e_3, e_4] = -e_3, \quad [e_2, e_5] = e_2, \quad [e_3, e_5] = -e_3.
\]
It is unimodular. For $\beta = e_1^* + e_2^* + e_3^*$, direct calculation gives
\[
    d\beta = -e_1^* \wedge e_4^* - e_2^* \wedge e_5^* + e_3^* \wedge e_4^* + e_3^* \wedge e_5^*,
\]
\[
    Z_\beta = \frac{1}{3}(e_1 + e_2 + e_3), \quad \beta \wedge (d\beta)^2 \neq 0.
\]
Consequently, for any integrable $\alpha = \sum_{j=1}^5 a_j e_j^*$, Proposition \ref{prop:unimodular_theta} gives
\[
    f_{\alpha, \beta} = a_1 + a_2 + a_3.
\]
Thus, among the integrable forms, $\alpha$ is an affine origin precisely when $a_1 + a_2 + a_3 \geq 0$. The integrability hypothesis is essential: the inequality alone does not assert that an arbitrary $1$-form defines a foliation.
\end{example}

\subsection{A non-unimodular example with zero Reeb contraction}

\begin{example}[A positive modular contribution]
Let $\g = \mathfrak{aff}(\R) \oplus \R e_0$, with basis $e_0, e_1, e_2$, brackets $[e_1, e_2] = e_2$, and $e_0$ central. The modular character is $\tau = e_1^*$. Fix $s \neq 0$ and set
\[
    \beta = e_2^* + s e_0^*, \quad \alpha = e_2^*.
\]
Then $d\beta = -e_1^* \wedge e_2^*$ and $\beta \wedge d\beta = -s e_0^* \wedge e_1^* \wedge e_2^* \neq 0$, so $\beta$ is contact. Its Reeb vector is $Z_\beta = s^{-1}e_0$, and hence $\alpha(Z_\beta) = 0$. The form $\alpha$ is integrable because $\alpha \wedge d\alpha = 0$. For $n = 1$, formula \eqref{eq:theta_modular} yields
\[
    \Theta(\alpha, \beta) = -\tau \wedge \alpha \wedge \beta = -s e_1^* \wedge e_2^* \wedge e_0^* = \vol_\beta.
\]
Therefore $f_{\alpha, \beta} = 1 > 0$. By Theorem \ref{thm:criterion}, $\alpha$ is an affine origin even though $\alpha(Z_\beta) = 0$.
\end{example}

\begin{remark}
The preceding example demonstrates that the modular term can affect the sign criterion. It does not establish the stronger claim that there exists a non-unimodular contact Lie algebra and an integrable form with $\alpha(Z_\beta) > 0$ but $f_{\alpha, \beta} < 0$. We leave that question open in this paper rather than include an unverified example.
\end{remark}

\section{Lie groups and compact quotients}

A left-invariant form on a Lie algebra determines a left-invariant differential form on the corresponding connected simply connected Lie group. The algebraic affine criterion therefore applies pointwise to the resulting forms on the group.

If $\Gamma$ is a uniform lattice in a Lie group $G$, every left-invariant differential form on $G$ descends to the compact quotient $\Gamma \backslash G$. Conversely, such a compact quotient exists only if $G$ admits a uniform lattice.

\begin{theorem}[Compact solvmanifolds]
Let $M = \Gamma \backslash G$ be a compact quotient of a connected solvable Lie group, equipped with a left-invariant contact form induced by $\beta$. Then $G$ is unimodular. Consequently, for every left-invariant integrable form $\alpha$,
\[
    \alpha \in \mathcal{A}(\beta) \iff \alpha(Z_\beta) \geq 0.
\]
\end{theorem}
\begin{proof}
A connected solvable Lie group admitting a uniform lattice is unimodular; see, for example, the classical results on lattices in solvable Lie groups cited in \cite{milnor1976curvatures}. Thus $\tau = 0$, and Proposition \ref{prop:unimodular_theta} together with Theorem \ref{thm:criterion} gives the equivalence.
\end{proof}

\begin{remark}
The last theorem concerns left-invariant forms on compact solvmanifolds. It does not imply that all affine deformation phenomena on non-unimodular Lie groups disappear; rather, a non-unimodular solvable Lie group cannot itself admit a uniform lattice.
\end{remark}

\end{document}